\documentclass{article}

\usepackage[english]{babel}

\usepackage[a4paper,top=2cm,bottom=2cm,left=3cm,right=3cm,marginparwidth=1.75cm]{geometry}

\usepackage{amsmath, amssymb}
\usepackage{graphicx}
\usepackage[colorlinks=true, allcolors=blue]{hyperref}
\usepackage{mathtools}
\usepackage{amsfonts}
\usepackage{amsthm}
\usepackage{tikz-cd}
\usepackage{comment}

\newtheorem{thm}{Theorem}[section]

\newtheorem{lem}[thm]{Lemma}
\newtheorem{prop}[thm]{Proposition}

\newtheorem{exam}[thm]{Example}
\numberwithin{equation}{section}
\newcommand{\mycomment}[1]{}

\newcommand{\ZZ}{\mathbb Z}

\newcommand{\eps}{\varepsilon}

\newcommand{\tC}{\Tilde{C}}

\newcommand{\s}{\sigma}

\usepackage[textwidth=0.8in]{todonotes}
\DeclareMathOperator{\Aut}{Aut}
\DeclareMathOperator{\id}{id}

\newcommand{\unnumberedfootnote}[1]{%
  \begingroup
  \renewcommand{\thefootnote}{}%
  \footnotetext[0]{#1}%
  \endgroup
}

\title{A note on smooth quotients of Prym varieties}
\author{Anatoli Shatsila}
\begin{document}
\maketitle
\begin{abstract}

We study pseudoreflections of geometric origin on Prym varieties of \'etale double covers. We prove that if the genus of the base curve is $g \geq 4$ then every such pseudoreflection has order 2. We use this result to show that, for $g \geq 5$, a non-trivial finite group $G$ of automorphisms of geometric origin acting faithfully on the Prym $P$ with $P/G$ smooth must be isomorphic to either $\mathbb{Z}/2\mathbb{Z}$ or $(\mathbb{Z}/2\mathbb{Z})^2$. We also show that the latter case can occur only for $g \leq 7$. This sharpens results of Auffarth, Lahoz and Naranjo.

\end{abstract}

\unnumberedfootnote{\textit{2020 Mathematics Subject Classification:} 14H40, 14L30, 14H37}
\unnumberedfootnote{\textit{Key words and phrases:} Abelian variety, Prym variety, Pseudoreflection, Smooth quotient.}

\section{Introduction}

Let $f\colon \widetilde C\longrightarrow C$ be an \'{e}tale double cover of a smooth projective curve $C$ of genus $g$, with covering involution $\iota$, and let
$(P,\Xi)=P(\widetilde C/C)$ be its principally polarized Prym variety. By an automorphism of an abelian variety we always mean a group automorphism. A polarized automorphism of $P$ is said to be of \textit{geometric origin} if it is induced by an automorphism of $\widetilde C$ commuting with $\iota$. We call an automorphism of $P$ \textit{a pseudoreflection} if its differential on the tangent space \(T_0P\) fixes a hyperplane pointwise.

The authors of \cite{ALN} studied pseudoreflections of geometric origin on Prym varieties. They proved that, for $g\geq5$, the locus of Prym varieties possessing such a pseudoreflection is the union of three explicit irreducible families. They also showed that the order of a pseudoreflection of geometric origin is either $2$ or $4$ (see \cite[Proposition 3.3]{ALN}). Moreover, assuming that the induced action on the Prym is faithful and the quotient is smooth, they obtained a list of nine possible groups. However, it is not known whether all of them can act on the Pryms with smooth quotients \cite[Remark 5.2]{ALN}. 

The main goal of this paper is to sharpen these conclusions. Our first result eliminates order $4$ for $g \geq 4$.

\begin{thm}[Theorem \ref{thm31}]\label{thm:order-two}
Let $g\geq4$. Every pseudoreflection on $P(\widetilde C/C)$ of geometric origin has order $2$.
\end{thm}

Moreover, we show that $g \geq 4$ is sharp (see Example \ref{ex:order-four}). The main idea of the proof is a careful analysis of the ramification data of intermediate covers from the properties of pseudoreflections.  

Combining Theorem \ref{thm:order-two} with the list in \cite[Proposition 5.1]{ALN} and the classification in \cite[Theorem 3.5]{ALA} yields the following.

\begin{thm}[Theorem \ref{thm41}]\label{thm:main}
Let $g\geq5$ and let $G\leq \Aut(\widetilde C)$ be a finite group such that every element of $G$ commutes with $\iota$. Let
\[
\rho\colon G\longrightarrow \Aut(P,\Xi)
\]
be the induced representation. Assume that $\rho$ is injective and that $P/\rho(G)$ is smooth. If $G$ is nontrivial, then $$ G\cong \ZZ/2\ZZ \qquad\text{or}\qquad G\cong (\ZZ/2\ZZ)^2.$$
Moreover, if $g > 7$ then $G\cong \ZZ/2\ZZ$. 
Both groups occur for some $g \geq 5$.
\end{thm}

Groups containing elements of order 4 are eliminated by Theorem \ref{thm:order-two}. The situation with 2-groups is more delicate and requires additional arguments involving an analysis of the decomposition of the Prym and its theta divisor. 

\subsection*{Acknowledgements.} 
The author has been supported by the Polish National Science Center project number 2024/53/N/ST1/01634.

\section{Preliminaries}\label{sec:prelim}

\subsection{Prym varieties and geometric automorphisms}

Let $f:\tC\to C$ be an \'{e}tale double cover of a smooth projective complex curve of genus $g$. The Prym variety of the cover $f$ is defined as $$P=P(\widetilde C/C):=\ker\bigl(\operatorname{Nm}_f: J\tC\to JC\bigr)^0.$$
The Prym variety \(P\) carries a natural principal polarization $\Xi$ and we have \(\dim P = d = g-1\). The covering involution $\iota$ acts as $+1$ on $f^*JC$ and as $-1$ on $P$. On holomorphic differentials this gives
$$H^0(\widetilde C,K_{\widetilde C})
=H^0(\widetilde C,K_{\widetilde C})^+
\oplus H^0(\widetilde C,K_{\widetilde C})^-,$$
with
$$ H^0(P,\Omega_P^1)\cong H^0(\widetilde C,K_{\widetilde C})^-.$$
Let $$Z(\iota)=\{\sigma\in\Aut(\widetilde C):\sigma\iota=\iota\sigma\}.$$
Every automorphism $\sigma\in Z(\iota)$ preserves $P$ and induces a polarized automorphism of the Prym. We denote the resulting homomorphism by
\[
\rho\colon Z(\iota)\longrightarrow \Aut(P,\Xi).
\]
An automorphism of $P$ lying in the image of $\rho$ is said to be of geometric origin.

We will use the following result.

\begin{lem}[\cite{ALN}, Lemma 3.1]\label{lem:kernel}
Assume $g\geq5$. Then $|\ker\rho|\leq2$. Moreover, if $\operatorname{id}\neq\mu\in\ker\rho$, then $$g\bigl(\widetilde C/\langle\mu,\iota\rangle\bigr)\leq1.$$
\end{lem}

\subsection{Pseudoreflections}

Let a finite group $H$ act holomorphically on a complex manifold $X$. \textit{A pseudoreflection at $x\in X$} is an element of $H$ whose fixed locus passing through $x$ has pure codimension $1$. If $A$ is an abelian variety and $s$ is an automorphism, then $s$ is a pseudoreflection if and only if its differential on $T_0A$ fixes a hyperplane pointwise (note that in this case we do not need to specify the points $x$). This means that the eigenvalues of $s$ on $T_0A$, and, equivalently, on its dual $H^0(A,\Omega_A^1)$, are
$$1,\ldots,1,\lambda, \qquad \lambda\neq1.$$

The Chevalley-Shephard-Todd theorem says that $X/H$ is smooth if and only if for every $x \in X$ the group $\operatorname{Stab}_H(x)$ is generated by pseudoreflections. Thus, if $H$ is a group of automorphisms of an abelian variety $A$, then smoothness of $A/H$ implies that $H$ is generated by pseudoreflections.

We have the following result.

\begin{prop}[\cite{ALN}, Proposition 3.3]\label{thm:ALN-orders}
For $g\geq3$, every pseudoreflection on a Prym variety of geometric origin has order $2$ or $4$.
\end{prop}

We will eliminate the second possibility when $g\geq4$.

\section{Excluding order four}\label{sec:order4}

We now prove Theorem \ref{thm:order-two}.

\begin{thm}
\label{thm31}
Let $g\geq 4$. Every pseudoreflection on $P(\widetilde C/C)$ of geometric origin has order $2$.
\end{thm}

\begin{proof}
We start with the case $g \geq 5$. By Proposition \ref{thm:ALN-orders}, the only possible orders are $2$ and $4$. Take an order-$4$ pseudoreflection  $\alpha\in\Aut(P,\Xi)$ of geometric origin. Choose $\sigma\in Z(\iota)$ with $\rho(\sigma)=\alpha.$ By Lemma \ref{lem:kernel}, the intersection $\langle\sigma\rangle\cap\ker\rho$ has order at most $2$. Since $\alpha$ has order $4$, it follows that $\sigma$ has order $4$ or $8$.

First, we rule out the case in which $\sigma$ has order $4$. Since $\alpha=\rho(\sigma)$ is an order-$4$ pseudoreflection, its action on $$H^0(P,\Omega_P^1)\cong H^0(\widetilde C,K_{\widetilde C})^-$$
has eigenvalues $$\underbrace{1,\ldots,1}_{d-1},\zeta,
\ \ \  \zeta\in\{i,-i\},$$ where $d=\dim P=g-1\geq4$. In particular, $\rho(\sigma^2)$ has eigenvalues $1^{d-1},-1$, whereas $\rho(\iota)=-\id_P$. Hence $\sigma^2\neq\iota$, and therefore
$$G:=\langle\iota,\sigma\rangle\cong\ZZ/2\ZZ\times\ZZ/4\ZZ.$$
Set $B:=\tC/G$ and consider two intermediate quotient curves
$$C_{\sigma}:=\tC/\langle\sigma\rangle, \ \ \ C_{\iota\sigma}:=\tC/\langle\iota\sigma\rangle.$$

Let us compute their genera. A differential on $C_{\sigma}$ is a $\sigma$-invariant differential on $\widetilde C$. On the $\iota$-invariant part, the $\sigma$-invariants are precisely the $G$-invariants, hence have dimension $g(B)$. The Prym part contributes \(d-1\) as \(\rho(\s)\) is a pseudoreflection. Therefore, $$g(C_{\s})=g(B)+d-1.$$
For $C_{\iota\s}$, a Prym differential fixed by $\iota\sigma$ must satisfy $\sigma^*\omega=-\omega$, but $-1$ is not an eigenvalue of $\rho(\sigma)$. Hence only the $G$-invariant differentials contribute, and we get $g(C_{\iota\sigma}) = g(B)$.

Let $r_{\sigma}$ and $r_{\iota\sigma}$ be the numbers of branch points of the double covers $C_{\sigma} \to B$ and $C_{\iota\sigma}\to B$. The Riemann-Hurwitz formula gives $$r_{\sigma}=2(d-g(B)), \ \ \ r_{\iota\sigma}=2(1-g(B)),$$
so
$$r_{\sigma}-r_{\iota\sigma}=2d-2.$$
Therefore, at least $2d-2$ points of $B$ are ramified in $C_{\sigma}/B$ but unramified in $C_{\iota\sigma}/B$.

Let $b\in B$ be one such point and choose a point $\widetilde b\in \tC$ lying above $b$. We denote by $I_b:=\operatorname{Stab}_G(\widetilde b)
=\{g\in G : g(\widetilde b)=\widetilde b\}$
its stabilizer. Since $G$ is abelian, this subgroup does not depend on the
choice of $\widetilde b$. Equivalently, $I_b$ is the inertia
group of the cover $\tC\to B$ at $b$. Since $C_{\iota\sigma}/B$ is unramified at $b$, we have $I_b\subseteq\langle\iota\sigma\rangle.$
On the other hand, since $C_{\sigma}/B$ is ramified at $b$, we have $I_b\not\subseteq\langle\sigma\rangle.$
The cyclic group $\langle\iota\sigma\rangle$ has order $4$, and its unique nontrivial proper subgroup is $\langle\sigma^2\rangle\subseteq\langle\sigma\rangle$. Hence we must have $I_b=\langle\iota\sigma\rangle\cong\ZZ/4\ZZ.$ Thus, there are at least $2d-2$ branch points of $\tC/B$ with this inertia group.

Now set
$$ C_{\iota\sigma^2}:=\tC/\langle\iota\sigma^2\rangle, \ \ \ C_{\iota,\sigma^2}=\tC/\langle\iota,\sigma^2\rangle.$$
Then $C_{\iota\sigma^2}\to C_{\iota,\sigma^2}$ is a double cover. We claim that $g(C_{\iota\sigma^2})-g(C_{\iota,\s^2})=1.$ Indeed, on the $\iota$-invariant differentials the invariance conditions for $\langle\iota\sigma^2\rangle$ and $\langle\iota,\sigma^2\rangle$ coincide. On the Prym part, invariance under $\iota\sigma^2$ is equivalent to $(\sigma^2)^*\omega=-\omega$. Since $\rho(\sigma)$ is a pseudoreflection, this is a one-dimensional condition. On the other hand, no Prym differential is invariant under $\iota$, so the claim is established.

Each point $b$ satisfying $I_b = \langle \iota\s \rangle$ gives a ramification point of $C_{\iota\sigma^2}\to C_{\iota, \sigma^2}$. Indeed, the fiber over $b$ of the cover $\tC \to B$ consists of two points $p_1$ and $p_2$ which are glued under $\tC \to C_{\iota\s^2}$ because $\langle \iota\s \rangle \cap \langle \iota\s^2\rangle = \{1\}$. Since the fiber of $C_{\iota\s^2} \to B$ consists of a unique point, this point must be a branch point of $C_{\iota\s^2} \to C_{\iota,\s^2}$.
Therefore, if $r$ is the ramification degree of $C_{\iota\s^2}\to C_{\iota,\s^2}$, then $r \geq 2d-2$.

On the other hand, by the Riemann-Hurwitz formula, we have $$r=4-2g(C_{\iota, \s^2})\leq4.$$
But $d\geq4$, so we get a contradiction with $r \geq 2d-2$. Thus, a lift of order $4$ is impossible.

It remains to rule out a lift of order $8$. Suppose that $\sigma$ has order $8$ and set $\tau:=\sigma^2$. Then $\tau$ has order $4$, while $\rho(\tau)=\alpha^2$
is a pseudoreflection of order 2. Moreover, $\tau^2=\sigma^4\neq1$
belongs to $\ker\rho$. By Lemma \ref{lem:kernel},
$$g\bigl(\tC/\langle\iota,\tau^2\rangle\bigr)\leq1.$$ Moreover, $\tau^2\neq\iota$, since
$\rho(\tau^2)=1$, whereas $\rho(\iota)=-\id_P$. Hence $\tau$ induces an automorphism of order 4 on $C$. Thus, $\tau$ satisfies the hypotheses of \cite[Lemma 3.5]{ALN}, which implies $g\leq3$. This contradicts $g\geq5$.

It remains to consider \(g=4\). Then
$d=\dim P=3$ and $g(\widetilde C)=7.$  Let
$\alpha\in\Aut(P,\Xi)$ be an order-$4$ pseudoreflection of geometric origin, and let $\sigma\in Z(\iota)$ be a lift of $\alpha$. We first observe that $ |\ker\rho|\leq3.$ Indeed, if $h=g(\widetilde C/\ker\rho)$, then, as in the proof of
\cite[Lemma 3.1]{ALN}, the fact that $\ker\rho$ acts trivially on $P$
implies $h\geq\dim P=3$. Hence the  Riemann-Hurwitz formula gives
$$
12=2g(\tC)-2\geq |\ker\rho|(2h-2)\geq4|\ker\rho|,$$ and the claim follows.

Therefore, $\sigma$ has order 4, 8 or 12. If $\sigma$ has order 12, then $\sigma^3$ has order 4, while
$\rho(\sigma^3)=\alpha^3$ is again an order-4 pseudoreflection. Thus, it suffices to exclude lifts of orders 4 and 8.

Suppose first that $\sigma$ has order 4. As in the proof above,

$$G:=\langle\iota,\sigma\rangle
\cong\mathbb{Z}/2\mathbb{Z}\times\mathbb {Z}/4\mathbb{Z}.$$
Set $B=\tC/G$. The same computation gives
$$g(C_\sigma)=g(B)+2,\ \ \ g(C_{\iota\sigma})=g(B),$$
and therefore
$r_\sigma=2(3-g(B))$ and $ r_{\iota\sigma}=2(1-g(B)).$
In particular, there are at least 4
branch points of $\tC\to B$ whose inertia group is $\langle\iota\sigma\rangle$ of order 4.

We claim that $g(B)=0$. Indeed, $r_{\iota\sigma}\geq0$, so
$g(B)\leq 1$. If $g(B)=1$, the four branch points above already
contribute $$4\cdot |G|\left(1-\frac14\right)=24$$
to the Riemann-Hurwitz formula for $\tC\to B$, while $2g(\widetilde C)-2=12$, which is a contradiction. Thus $B\simeq\mathbb{P}^1$.

Applying the Riemann-Hurwitz formula to 
$\tC\to\mathbb P^1$, we obtain $$12=-16+\sum_{b\in\mathbb P^1}
8\left(1-\frac1{|I_b|}\right).$$
Hence the total ramification contribution is 28. The four points
with inertia group $\langle\iota\sigma\rangle$ contribute 24. It follows that there are exactly four such points and exactly
one additional branch point, whose inertia group has order 2.
Thus the $G$-cover has signature $(0;4,4,4,4,2).$ 

Consider the local monodromies of this cover. The first four lie in
the cyclic subgroup $H:=\langle\iota\sigma\rangle\subsetneq G.$
Since the product of the five local monodromies is the identity, the
fifth monodromy also lies in $H$. Hence all local monodromies belong
to $H$. On the other hand, since the cover is connected and the base
is $\mathbb P^1$, its local monodromies must generate $G$. This is
a contradiction.

It remains to exclude the case in which $\sigma$ has order 8. Let $\tau=\sigma^2$. Then $\tau$ has order $4$, $\rho(\tau)=\alpha^2$ is a
pseudoreflection of order 2, and $\mu:=\tau^2=\sigma^4$ is a nontrivial involution acting trivially on $P$. We claim that $g\bigl(\tC/\langle\iota,\mu\rangle\bigr)\leq1.$ Indeed, the Prym variety of the double cover $\tC/\langle\iota\mu\rangle \to \tC/\langle\iota,\mu\rangle$
corresponds to the subspace of Prym differentials on which $\mu$
acts as $-1$. Since $\mu\in\ker\rho$, this subspace is zero. Thus, by the Riemann-Hurwitz formula, $g(\tC/\langle\iota,\mu\rangle) \leq 1$. Moreover, $\tau^2\neq\iota$, since
$\rho(\tau^2)=1$, whereas $\rho(\iota)=-\id_P$. Hence $\tau$ induces an automorphism of order 4 on $C$ so it satisfies all the hypotheses of \cite[Lemma 3.5]{ALN}. That lemma gives $g\leq3$, contradicting $g=4$.

Therefore an order-4 pseudoreflection cannot occur for $g=4$ either, and the result follows.
\end{proof}

\begin{exam}
\label{ex:order-four}
The bound in Theorem \ref{thm:order-two} is sharp. Indeed, there exist étale double covers of curves of genus 3 whose Prym varieties admit a pseudoreflection of geometric origin of order 4.

Let $$\widetilde G=\langle \iota,\sigma\rangle
\cong \mathbb Z/2\mathbb Z\times\mathbb Z/4\mathbb Z,$$
where $\iota$ has order 2 and $\sigma$ has order 4. Consider a connected $\widetilde G$-Galois cover $\tC\longrightarrow \mathbb{P}^1$ with local monodromies $$\iota\sigma,\ \ \ \iota\sigma,\ \ \ \sigma^2,\ \ \ \iota\sigma^2,\ \ \ \iota\sigma^2.$$
Such a cover exists by the Riemann existence theorem: these elements generate $\widetilde G$ and their product is the identity. Their orders are $(4,4,2,2,2)$, respectively. Hence the Riemann-Hurwitz formula gives
$$2g(\tC)-2 = -16+2\cdot 8\left(1-\frac14\right)+3\cdot 8\left(1-\frac12\right)=8,$$ and hence $g(\tC)=5$.

None of the inertia groups contains $\iota$, so $\iota$ acts freely on $\tC$. Thus $C:=\tC/\langle\iota\rangle$
has genus 3, and $\widetilde C\to C$ is an étale double cover. Let $P=P(\widetilde C/C)$. We claim that the automorphism of $P$ induced by $\sigma$ is a pseudoreflection of order $4$. Set $B=\tC/\widetilde G\simeq\mathbb P^1$. The quotient $\tC/\langle\sigma\rangle\longrightarrow B$
is a double cover branched at four points, and hence $g\bigl(\tC/\langle\sigma\rangle\bigr)=1$. On the other hand, $\tC/\langle\iota\sigma\rangle\to B$
is a double cover branched at two points, so
$g\bigl(\tC/\langle\iota\sigma\rangle\bigr)=0$. Let $$V=H^0(P,\Omega_P^1)=H^0(\widetilde C,\omega_{\tC})^{-}_{\iota}.$$
Since $g(B)=0$, we have $\dim V^{\sigma}= g(\tC/\langle \sigma \rangle) - g(B) =1.$
Moreover, on $V$ the involution $\iota$ acts as $-1$, and hence the $(-1)$-eigenspace of $\sigma$ is precisely the subspace on which $\iota\sigma$ acts trivially. Therefore, $$\dim V^{-}_{\sigma}=g\bigl(\tC/\langle\iota\sigma\rangle\bigr)-g(B)=0.$$
Since $\dim V=2$ and $\sigma^4=1$, the eigenvalues of $\sigma$ on $V$ are $1$ and $\pm i$, hence it is a pseudoreflection of order 4.
\end{exam}

\section{The faithful smooth-quotient classification}\label{sec:groups}

We recall from [ALA22, Theorem 3.5] that if \((P, \Xi)\) is a principally polarized abelian variety and \(H \leq \operatorname{Aut}(P, \Xi)\) is such that \(P / H\) is smooth, then there is an isogeny \(P \simeq\) \(X \times Y\) where \(H\) acts trivially on \(X\), and on \(Y\) with finitely many fixed points. Moreover, there exist elliptic curves \(E_1, \ldots, E_r, F_1, \ldots, F_s\) such that
\[
Y \cong \prod_{i=1}^r E_i^{n_i} \times \prod_{j=1}^s F_j^{t_j}
\]

\[
H=\prod_{i=1}^r\bigl(\left(\mathbb{Z} / m_i \mathbb{Z}\right)^{n_i} \rtimes S_{n_i}\bigr) \times \prod_{j=1}^s S_{t_j+1}
\]
where \(m_i \in\{2,3,4,6\}, n_i, t_j \geq 1\), and each factor of \(H\) acts on the corresponding factor of \(Y\) in an explicit way.

We now prove Theorem \ref{thm:main}, the main result of the paper.

\begin{thm}
\label{thm41}
Let $g\geq5$ and let $G\leq \Aut(\widetilde C)$ be a finite group such that every element of $G$ commutes with $\iota$. Let
$$\rho: G\longrightarrow \Aut(P,\Xi)$$
be the induced representation. Assume that $\rho$ is injective and that $P/\rho(G)$ is smooth. If $G$ is nontrivial, then $$G\cong \ZZ/2\ZZ \ \ \ \text{or}\ \ \ G\cong (\ZZ/2\ZZ)^2.$$
Moreover, if $g > 7$ then $G\cong \ZZ/2\ZZ$. 
Both groups occur for some $g \geq 5$.
\end{thm}

\begin{proof}
In the proof we identify $G$ with $\rho(G)$. According to \cite[Proposition 5.1]{ALN}, there are nine possible groups: $$\ZZ/2\ZZ, (\ZZ/2\ZZ)^2, \ZZ/4\ZZ, (\ZZ/4\ZZ)^2, \ZZ/2\ZZ \times \ZZ/4\ZZ, (\ZZ/2\ZZ)^3, (\ZZ/2\ZZ)^2\rtimes S_2, (\ZZ/4\ZZ)^2\rtimes S_2, \bigl((\ZZ/2\ZZ)^2\rtimes S_2\bigr)\times S_2.$$

We eliminate the seven groups other than $\ZZ/2\ZZ$ and $(\ZZ/2\ZZ)^2$.

\medskip
\noindent\textit{Step 1: Groups involving $4$-torsion.}

Since $P/G$ is smooth, $G$ is generated by pseudoreflections. By Theorem \ref{thm31}, all of these generators are involutions. This eliminates $\ZZ/4\ZZ, (\ZZ/4\ZZ)^2, \ZZ/2\ZZ\times \ZZ/4\ZZ$, since these groups are not generated by involutions. The only other possible group containing $4$-torsion is 
$G=(\ZZ/4\ZZ)^2\rtimes S_2,$ where $S_2$ exchanges the two $\ZZ/4$ factors. The homomorphism
$$\varphi\colon G\longrightarrow\ZZ/4\ZZ, \ \ \  \varphi(a,b;\eps)=a+b$$
is surjective. If $x\in G$ is an involution, then
\[
2\varphi(x)=\varphi(x^2)=0,
\]
so $\varphi(x)\in\{0,2\}$. Hence the subgroup generated by all involutions has image contained in the proper subgroup $\{0,2\}\subset\ZZ/4\ZZ$. Thus, $G$ is not generated by involutions. We have eliminated all four candidates involving $\ZZ/4\ZZ$.

\medskip
\noindent\textit{Step 2: The dihedral cases.}

Suppose first that $G\cong (\ZZ/2\ZZ)^2\rtimes S_2$ is the dihedral group of order 8. The group $G$ is not a direct product of two groups and $|G| = 8$, hence by \cite[Theorem 3.5]{ALA}, we have $$(P,\Xi)\cong(E^2,\Theta_2)\times(B,\Theta_B)$$ as principally polarized abelian varieties. Here, $E$ is an elliptic curve, with $G$ acting by the two coordinate sign changes and by permutation of the two factors on $E^2$ and trivially on $B$. Since $d\geq4$, we have $\dim B=d-2>0$, therefore $\Xi$ has at least three irreducible components.

Moreover, since $P$ is decomposable, it follows from \cite[Remark 4.2]{ALN} that
$$(P,\Xi)\cong(F,[0])\times(JH,\Theta_H)$$ for an elliptic curve $F$ and a hyperelliptic curve $H$. But $\Theta_H$ is irreducible, hence $\Xi$ has two irreducible components, which is a contradiction. Hence $(\ZZ/2\ZZ)^2\rtimes S_2$ cannot occur. The same argument excludes $\bigl((\ZZ/2\ZZ)^2\rtimes S_2\bigr)\times S_2$.

\medskip
\noindent\textit{Step 3: The group $(\ZZ/2\ZZ)^3$.}

It remains to rule out $G \cong (\ZZ/2\ZZ)^3$. Let $\sigma_1, \sigma_2, \sigma_3$ be pseudoreflections generating $G$. Then we have $$H^0(P,\Omega_P^1)=W\oplus L_1\oplus L_2\oplus L_3,$$
where $\dim W=d-3, \dim L_i=1$ and $\sigma_i$ acts by $-1$ on $L_i$ and by $+1$ on the other summands. Since $d\geq4$, the common fixed space $W$ is nonzero. Every element of $G$ acts trivially on $W$, whereas $\iota$ acts as $-1$ on all Prym differentials, hence $\iota\notin G$. Let $$K:=\langle\iota,G\rangle\cong(\ZZ/2\ZZ)^4.$$
Set $B:=\widetilde C/K,$ and consider the eight characters $$\chi: K\longrightarrow\{\pm1\}$$
with $\chi(\iota)=-1$. For each such character, let $C_\chi:=\tC/\ker\chi$. If $V_\chi$ denotes the $\chi$-eigenspace in $H^0(\tC,K_{\tC})$, then $$ H^0(C_\chi,K_{C_\chi})=H^0(B,K_B)\oplus V_\chi,$$
so $m_{\chi} := \dim V_\chi=g(C_\chi)-g(B).$

Let $n_\chi$ be the number of branch points of $C_\chi\to B$. The Riemann-Hurwitz formula gives $$n_\chi=2(m_\chi-g(B)+1).$$ Let $R$ be the number of branch points of the full $K$-cover $\widetilde C\to B$. Recall that for the action of a group on a smooth curve all inertia subgroups are cyclic, thus, since $K$ is elementary abelian, every nontrivial inertia subgroup has order $2$. Its generator is never $\iota$, because the covering involution is fixed-point-free. For any nontrivial inertia element $v\neq\iota$, exactly four of the eight characters with $\chi(\iota)=-1$ satisfy $\chi(v)=-1$. Thus, every branch point of $\tC/B$ is counted in exactly four of the double covers $C_\chi/B$, and
$$\sum_\chi n_\chi=4R.$$
Thus, we obtain $$R=\frac{d+8-8g(B)}{2}.$$ In particular, $d$ is even.

Now take the character $\chi_0$, which is trivial on $G$, i.e. corresponding to the common fixed space $W$. Its eigenspace has dimension $d-3$, and therefore $n_{\chi_0}=2(d-g(B)-2).$
Every branch point of the corresponding double cover is a branch point of $\tC/B$, so $n_{\chi_0}\leq R$. Therefore,

$$2(d-g(B)-2)\leq\frac{d+8-8g(B)}{2},$$
which simplifies to

$$3d+4g(B)\leq16.$$
Since $d\geq4$ is even, we must have $d = 4$.

By \cite[Theorem 3.5 and Lemma 3.6]{ALA}, each of the three $\ZZ/2\ZZ$-factors gives either a principally
polarized elliptic direct factor, or a one-dimensional factor of the standard
construction, in which case the restricted polarization is of type $(2)$.
Let $p$ and $q$ denote the numbers of factors of these two types,
respectively. 

Again, by \cite[Remark 4.2]{ALN}, the only case in which a Prym variety admitting
a geometric pseudoreflection is a polarized product is $$(P,\Xi)\simeq (E,\Theta_E)\times(JH,\Theta_H).$$
Therefore, $p\leq 1$, and thus $q\geq2$.

For $q$ one-dimensional standard factors, the moving part has dimension
$q$ and polarization type $(2,\ldots,2)$. In the standard construction from \cite{ALA}, the auxiliary polarized abelian variety has dimension at least $q$. Consequently, $$d\ge p+2q.$$
Since $p+q=3$ and $p\leq 1$, we obtain
$d \geq 5$
contradicting \(d=4\). Thus \(G\simeq(\ZZ/2\ZZ)^3\) cannot occur.

\medskip
\noindent\textit{Step 4: $G \cong (\mathbb{Z}/2\mathbb{Z})^2$ forces $g \leq 7$.}

Let $G \cong (\mathbb{Z}/2\mathbb{Z})^2$ and $B := \tC / (G \times \langle\iota\rangle)$. Repeating the analysis from the first part of Step 3, we get the inequality $$2(g - g(B) - 2) \leq d - 4g(B) + 4,$$ which is equivalent to $$d+2g(B) \leq 6.$$ Thus, $d = g-1 \leq 6$, hence $g \leq 7$.

Finally, we note that both groups $\ZZ/2\ZZ$ and $(\ZZ/2\ZZ)^2$ occur by the constructions in \cite{ALN}; see \cite[Remark 5.2]{ALN} and also \cite[Remark 2.7]{BO26}. This proves the theorem.
\end{proof}

\subsection*{Statement on AI use}
A substantial part of this work was developed with the assistance of ChatGPT 5.6 Sol. In particular, the statement of Theorem \ref{thm:order-two} and the main ideas behind its proof were suggested by the LLM and served as the starting point for this work. The full text of the article was written by the author, who takes full responsibility for the correctness of all arguments.

\textsc{A. Shatsila, Doctoral School of Exact and Natural Sciences, Jagiellonian University, ul. prof. Stanisława Łojasiewicza 6, 30-348 Kraków, Poland}\\
\textit{email address:} anatoli.shatsila@doctoral.uj.edu.pl.


\begin{thebibliography}{99}

\bibitem{BO26}
P.~Borówka and A.~Ortega,
\emph{Klein coverings over hyperelliptic genus 3 curves},
arXiv:2602.18969, 2026.

\bibitem{ALA}
R.~Auffarth and G.~Lucchini Arteche,
\emph{Smooth quotients of principally polarized abelian varieties},
Mosc. Math. J. \textbf{22} (2022), no.~2, 225--237.

\bibitem{ALN}
R.~Auffarth, M.~Lahoz and J.~C.~Naranjo,
\emph{Pseudoreflections on Prym varieties},
arXiv:2412.04940v2.

\end{thebibliography}
\end{document}